\documentclass[a4paper,english,fontsize=11pt,parskip=half,abstracton]{scrartcl}
\usepackage{babel}
\usepackage[utf8]{inputenc}
\usepackage[T1]{fontenc}
\usepackage[a4paper,left=20mm,right=20mm,top=30mm,bottom=30mm]{geometry}
\usepackage{amsmath}
\usepackage{amsthm}
\usepackage{amssymb}
\usepackage{enumerate}
\usepackage{aliascnt}
\usepackage[bookmarks=true,
            pdftitle={On a theorem of Ledermann and Neumann},
            pdfauthor={Benjamin Sambale},
            pdfkeywords={automorphism groups},
            pdfstartview={FitH}]{hyperref}

\newtheorem*{ThmA}{Theorem~A}

\newaliascnt{Lem}{Thm}
\newtheorem{Lem}[Lem]{Lemma}
\aliascntresetthe{Lem}
\newaliascnt{Prop}{Thm}
\newtheorem{Prop}[Prop]{Proposition}
\aliascntresetthe{Prop}
\newaliascnt{Cor}{Thm}

\aliascntresetthe{Cor}
\newaliascnt{Con}{Thm}

\aliascntresetthe{Con}
\theoremstyle{definition}
\newaliascnt{Def}{Thm}

\aliascntresetthe{Def}
\newaliascnt{Ex}{Thm}

\aliascntresetthe{Ex}

\numberwithin{equation}{section}

\allowdisplaybreaks[1]

\renewcommand{\phi}{\varphi}

\newcommand{\Z}{\mathrm{Z}}
\newcommand{\ZZ}{\mathbb{Z}}

\newcommand{\FF}{\mathbb{F}}

\newcommand{\Aut}{\mathrm{Aut}}
\newcommand{\End}{\mathrm{End}}

\newcommand{\GL}{\mathrm{GL}}

\mathchardef\ordinarycolon\mathcode`\:  %defines a nice ":=" 
 \mathcode`\:=\string"8000
 \begingroup \catcode`\:=\active
   \gdef:{\mathrel{\mathop\ordinarycolon}}
 \endgroup

\title{On a theorem of Ledermann and Neumann}
\author{Benjamin Sambale\footnote{Institut für Mathematik, Friedrich-Schiller-Universität Jena, 07737 Jena, Germany, 
\href{mailto:benjamin.sambale@uni-jena.de}{benjamin.sambale@uni-jena.de}}}
\date{\today}

\begin{document}
\frenchspacing
\maketitle
\begin{abstract}\noindent
We give a short and self-contained proof of a theorem of Ledermann and Neumann stating that there are only finitely many finite groups with a given number of automorphisms. We also discuss the history of related conjectures.
\end{abstract}

%\tableofcontents

\section{Introduction}
Obviously, every finite group $G$ has only finitely many automorphisms. In fact, 
\begin{equation}\label{easy}
|\Aut(G)|\le(|G|-1)!
\end{equation}
as every automorphism permutes the non-trivial elements of $G$ (an optimal bound will be given at the end of the paper). 

It is far less obvious, if conversely the order of $G$ is bounded by a function depending only on $|\Aut(G)|$. 
Ledermann and Neumann~\cite[Theorem~6.6]{LedermannNeumann} affirmatively answered this question in 1956 by constructing an explicit (but crude) bound. Unfortunately, their proof is rather long and complicated. In a second paper~\cite[Theorem~8.6]{LedermannNeumann2} the authors provided a local version by bounding the $p$-part $|G|_p$ in terms of $|\Aut(G)|_p$ where $p$ is a prime (this resolved a conjecture of Scott~\cite{ScottAut} and is now presented in the recent book~\cite[Chapter~3]{Passi}). Ledermann and Neumann's original theorem was rediscovered by Nagrebeckiĭ~\cite{Nagrebeckii} in 1970 and (presumably) independently by Iyer~\cite[Theorem~3.1]{Iyer} in 1979. The former proof is somewhat opaque and the latter implicitly relies on \cite{LedermannNeumann2} via the PhD thesis of Hyde~\cite{Hyde}. However, Nagrebeckiĭ~\cite[Theorem~4]{Nagrebeckii2} gave a more transparent second proof within a generalized framework dealing with infinite groups. It seems that his work was not widely recognized (the English translation is not mentioned on MathSciNet for instance).
The purpose of the present paper is to give a self-contained proof of the following version of the Ledermann--Neumann theorem based on some ideas from \cite{Nagrebeckii2}.

\begin{ThmA}
For every integer $n$ there exist only finitely many finite groups with at most $n$ automorphisms.
\end{ThmA}

Our proof of Theorem~A uses only first principles of elementary group theory, which are summarized in the next section.
In the final section we discuss some related conjectures. The reader interested in infinite groups can find several generalizations of Theorem~A in \cite{AlperinAut,Nagrebeckii3,Nagrebeckii4,RobinsonAut,RobinsonAut2,RobinsonAut3}.

\section{Preliminaries}

All groups considered in this paper are finite. 
Every element $g$ of a group $G$ induces an \emph{inner} automorphism $f_g$ of $G$ by sending $x$ to $gxg^{-1}$. The map $G\to\Aut(G)$, $g\mapsto f_g$ is a homomorphism whose kernel is the \emph{center} $\Z(G)=\{g\in G:gx=xg\,\forall x\in G\}$ of $G$. In particular,
\begin{equation}\label{inn}
|G/\Z(G)|\le|\Aut(G)|
\end{equation}
by the first isomorphism theorem.

For $x,y\in G$ we define the \emph{commutator} $[x,y]:=xyx^{-1}y^{-1}\in G$. A direct computation reveals
\begin{equation}\label{comm}
g[x,y]g^{-1}=[gxg^{-1},gyg^{-1}],\quad[x,y^2]=[x,y]y[x,y]y^{-1}=[x,y][yxy^{-1},y]
\end{equation}
for $g\in G$. The commutators of $G$ generate the \emph{commutator subgroup} $G'$ of $G$. By \eqref{comm}, $G'$ is normal in $G$ and $G/G'$ is abelian.

The \emph{exponent} $\exp(G)$ of $G$ is the smallest positive integer $e$ such that $g^e=1$ for all $g\in G$. Clearly, the exponent of every subgroup or quotient of $G$ divides $\exp(G)$. %By Lagrange's theorem $\exp(G)$ is a divisor of $|G|$.
The smallest integer $d$ such that $G$ can be generated by $d$ elements is denoted by $d(G)$.

Now assume that $G$ is abelian. Then clearly
\begin{equation}\label{dG}
|G|\le\exp(G)^{d(G)}.
\end{equation}
By the main theorem of finite abelian groups there exists a decomposition 
%into a direct product of cyclic groups
\begin{equation}\label{decomp}
G=\langle x_1\rangle\times\ldots\times\langle x_k\rangle
\end{equation}
such that the order of $x_i$ is a prime power for $i=1,\ldots,k$. 
This yields a factorization into \emph{primary components} $G=G_{p_1}\times\ldots\times G_{p_n}$ where $p_1,\ldots,p_n$ are the prime divisors of $|G|$ and $G_{p_i}$ is the set of $p_i$-elements of $G$ for $i=1,\ldots,n$.
Suppose that $x_1$ in \eqref{decomp} is a $p$-element and $r\in\ZZ$ is a primitive root modulo $p$. Then the map $x_1\mapsto x_1^r$ defines an automorphism $\alpha$ of $\langle x_1\rangle$ whose order is divisible by $p-1$. Since $\alpha$ extends to $G$, we obtain
\begin{equation}\label{aut}
p-1\le|\Aut(G)|
\end{equation}
whenever $p$ divides $|G|$.

%where $q_1,\ldots,q_s$ are prime powers.
% and $n_1,\ldots,n_d$ are uniquely determined integers such that $n_1\mid n_2\mid\ldots\mid n_d$. Here $d=:d(G)$ is the smallest integer such that $G$ can be generated by $d$ elements. Obviously,
%\begin{equation}\label{dG}
%|G|\le\exp(G)^{d(G)}.
%\end{equation}
%The first decomposition in \eqref{decomp} shows that $G$ is the direct product of its \emph{primary components} $G=G_{p_1}\times\ldots\times G_{p_n}$ where $p_1,\ldots,p_n$ are the prime divisors of $|G|$ and $G_{p_i}$ is the set of $p_i$-elements of $G$ for $i=1,\ldots,n$.
%Let $r\in\ZZ$ be a primitive root modulo $p_i$. Then the map sending $x$ to $x^r$ is an automorphism $\alpha$ of $G_{p_i}$ of order divisible by $p_i-1$. Since $\alpha$ extends to $G$ we obtain
%\begin{equation}\label{aut}
%p_i-1\le|\Aut(G)|
%\end{equation}
%for $i=1,\ldots,n$.

%It is easy to see that $G_{p_i}^{p_i}:=\{x^{p_i}:x\in G_{p_i}\}$ is a normal subgroup such that $G_{p_i}/G_{p_i}^{p_i}$ is \emph{elementary abelian} (i.\,e. $\exp(G_{p_i}/G_{p_i}^{p_i})=p_i$) and $|G_{p_i}/G_{p_i}^{p_i}|=p^{d(G_{p_i})}$. Moreover,
%\[\Phi(G)=G_{p_1}^{p_1}\times\ldots\times G_{p_n}^{p_n}\]
%is called the \emph{Frattini subgroup} of $G$.

Finally we need a rather special case of the famous Schur--Zassenhaus theorem, which is at the same time a special case of Burnside's transfer theorem.

\begin{Prop}\label{SZ}
Let $p$ be a prime such that $|G|_p=|\Z(G)|_p$. Then $G=\Z(G)_p\times Q$ for some $Q\le G$.
\end{Prop}
\begin{proof}
See \cite[Theorem~3.3.1 or Theorem~7.2.1]{Kurzweil}.
%Since $\Z(G)_p$ is an abelian normal Hall subgroup of $G$, the result follows from \cite[Theorem~3.3.1 or Theorem~7.2.1]{Kurzweil}.
\end{proof}

\section{Proof of Theorem~A}

In the following let $G$ be a finite group and $n:=|\Aut(G)|$. We prove Theorem~A by bounding $|G|$ in terms of $n$. 
%Although it is possible to extract an explicit bound
This is done in a series of lemmas.

\begin{Lem}[Schur~\cite{Schur}]\label{schur}
$|G'|\le n^{2n^3}$.
\end{Lem}
\begin{proof}[Proof \textup{(}\textnormal{Rosenlicht~\cite{Rosenlicht}}\textup{)}.]
Let $g_1,\ldots,g_m\in G$ be representatives for the cosets of $G/\Z(G)$. Then $m=|G/\Z(G)|\le n$ by \eqref{inn}. 
Arbitrary elements $g,h\in G$ can be written as $g=g_iz$ and $h=g_jw$ with $z,w\in\Z(G)$. It follows that $[g,h]=[g_i,g_j]$. Hence, the set of commutators
\[\Gamma:=\bigl\{[g,h]:g,h\in G\bigr\}=\bigl\{[g_i,g_j]:1\le i,j\le m\bigr\}\]
has at most $m^2$ elements. It suffices therefore to show that every element $g\in G'$ is a product of at most $m^3$ commutators. 
Let $g=\gamma_1\ldots \gamma_s$ such that $\gamma_1,\ldots,\gamma_s\in\Gamma$ and $s$ is as small as possible. By way of contradiction suppose that $s>m^3$. Then some commutator $\gamma=[x,y]$ appears more than $m$ times among the $\gamma_i$. Since $\gamma_i\gamma_{i+1}=\gamma_{i+1}\delta$ where $\delta:=\gamma_{i+1}^{-1}\gamma_i\gamma_{i+1}\in\Gamma$ by \eqref{comm}, we may assume that $\gamma=\gamma_1=\ldots=\gamma_{m+1}$. Since $\gamma^m=\gamma^{|G/\Z(G)|}\in\Z(G)$, we have
\[\gamma^{m+1}=\gamma\gamma^m=\gamma y\gamma^my^{-1}=\gamma(y\gamma y^{-1})^m=\gamma y\gamma y^{-1}\cdot (y\gamma y^{-1})^{m-1}=[x,y^2][yxy^{-1},y]^{m-1}\]
according to \eqref{comm}. But now $g=\gamma^{m+1}\gamma_{m+2}\ldots\gamma_s$ is a product of $s-1$ commutators. Contradiction.
\end{proof}

\autoref{schur} shifts the focus to the abelian group $G/G'$. It is however not clear if and how automorphisms of $G/G'$ lift to $G$.

\begin{Lem}\label{primes}
Every prime divisor $p$ of $|G|$ is at most $n+1$.
\end{Lem}
\begin{proof}
If $|G/\Z(G)|_p\ne 1$, then $p\le n$ by \eqref{inn}. Otherwise, $|\Z(G)|_p=|G|_p$ and $G=\Z(G)_p\times Q$ by \autoref{SZ}. 
Since every automorphism of $\Z(G)_p$ extends to $G$, we obtain $p-1\le n$ by \eqref{aut}.
\end{proof}

A careful analysis of the proof shows that $p^2\mid|G|$ implies $p\mid n$. This observation of Herstein--Adney~\cite{Herstein} is however not needed below.

\begin{Lem}\label{exp}
The exponent $\exp(G)$ is bounded in terms of $n$. 
\end{Lem}
\begin{proof}
By \autoref{schur} it suffices to show that $\exp(G/G')$ is bounded in terms of $n$.
By \eqref{decomp} we may write $G/G'=H/G'\times\langle gG'\rangle$ with $g\in G$ and $H\unlhd G$. 
Then $G=H\langle g\rangle$ and $H\cap\langle g\rangle\le G'$. Note that
\[N:=|G/\Z(G)|\cdot |G'|\cdot\prod_{p\,\mid\, |G|}p\le n\cdot n^{2n^3}\cdot (n+1)!\]
by \eqref{aut}, \autoref{schur} and \autoref{primes}.
Let $h_1,h_2\in H$ and $i,j\in\ZZ$ such that $h_1g^i=h_2g^j$. Then $h_2^{-1}h_1=g^{j-i}\in H\cap\langle g\rangle\le G'$. Since $|G'|$ divides $N$ we conclude that $h_2^{-1}h_1=(g^{j-i})^{1+N}$. Therefore the map
\[\alpha:G\to G,\quad hg^i\mapsto hg^{i(1+N)}\quad(h\in H,\,i\in\ZZ)\]
is well-defined. Since $g^N\in\langle g^{|G/\Z(G)|}\rangle\subseteq\Z(G)$, we obtain
\begin{align*}
\alpha(h_1g^ih_2g^j)&=\alpha(h_1(g^ih_2g^{-i})g^{i+j})=h_1(g^ih_2g^{-i})g^{(i+j)(1+N)}=h_1g^ih_2g^{-i+i(1+N)}g^{j(1+N)}\\
&=h_1g^{i+iN}h_2g^{j(1+N)}=\alpha(h_1g^i)\alpha(h_2g^j)
\end{align*}
for all $h_1,h_2\in H$ and $i,j\in\ZZ$. Hence, $\alpha$ is a homomorphism. Every prime divisor of $|\langle g\rangle|$ divides $|G|$ and is therefore coprime to $1+N$. Consequently, $\langle g^{1+N}\rangle=\langle g\rangle$ and $\alpha$ is surjective. Now $\alpha\in\Aut(G)$, since $G$ is finite. 
In particular, $g=\alpha^n(g)=g^{(1+N)^n}$. Since $\langle gG'\rangle$ was an arbitrary direct factor of $G/G'$, it follows that
\[\exp(G/G')\le (1+N)^n-1.\qedhere\]
\end{proof}

\begin{Lem}\label{dirpel}
Let $A$ be an abelian group and $a\in A$ of prime order $p$. Then there exists a decomposition $A=B\times C$ such that $B$ is cyclic and $a\in B$.
\end{Lem}
\begin{proof}
By \eqref{decomp} we may assume that $A=A_p$. 
%Let
%\[A=\langle x_1\rangle\times\ldots\times\langle x_n\rangle\] 
%such that $|\langle x_i\rangle|=p^{\alpha_i}$ and $\alpha_1\ge\ldots\ge \alpha_n$. Since $a$ has order $p$, there exists $\beta_1,\ldots,\beta_n\in\ZZ$ such that $a=x_1^{\beta_1p^{\alpha_1-1}}\ldots x_n^{\beta_np^{\alpha_n-1}}$.
%Define 
%\[b:=x_1^{\beta_1p^{\alpha_1-\alpha_n}}x_2^{\beta_2p^{\alpha_2-\alpha_n}}\ldots x_n^{\beta_n}\] 
%and $B:=\langle b\rangle$. Then $a=b^{p^{\alpha_n-1}}\in B$.
Let $A=B\times C$ such that $a\in B$ and $|B|$ is as small as possible ($B=A$ may do). Let 
\[B=\langle x_1\rangle\times\ldots\times\langle x_n\rangle\] 
such that $|\langle x_i\rangle|=p^{\alpha_i}$ and $\alpha_1\ge\ldots\ge \alpha_n$. The choice of $B$ implies that $a=x_1^{\beta_1p^{\alpha_1-1}}\ldots x_n^{\beta_np^{\alpha_n-1}}$ where $\beta_i\not\equiv 0\pmod{p}$ for $i=1,\ldots,n$.
We define
\[b:=x_1^{\beta_1p^{\alpha_1-\alpha_n}}x_2^{\beta_2p^{\alpha_2-\alpha_n}}\ldots x_n^{\beta_n}.\]
Then $a=b^{p^{\alpha_n-1}}\in\langle b\rangle$ and $B=\langle x_1\rangle\times\ldots\times\langle x_{n-1}\rangle\times\langle b\rangle$. Now the minimality of $B$ yields $B=\langle b\rangle$ as desired.
\end{proof}

\begin{Lem}\label{lemabel2}
Let $B\le A$ be abelian groups. Then there exists a decomposition $A=C\times D$ such that $B\le C$ and $d(C)\le|B|$.
\end{Lem}
\begin{proof}
We argue by induction on $|B|$. If $|B|=1$, then we take $C=1$ and $D=A$. Now assume that $|B|>1$ and pick a subgroup $B_0\le B$ of prime index $p$. By induction there exists a decomposition $A=C_0\times D_0$ such that $B_0\le C_0$ and $d(C_0)\le|B_0|$.
Let $b\in B\setminus B_0$ and write $b=cd$ with $c\in C_0$ and $d\in D_0$. Then 
\[d^p=b^pc^{-p}\in B_0C_0\cap D_0\le C_0\cap D_0=1.\] By \autoref{dirpel} there exists a decomposition $D_0=D_1\times D_2$ such that $D_1$ is cyclic and $d\in D_1$. Now we define $C:=C_0\times D_1$. Then $B=B_0\langle b\rangle\le C$, $A=C_0\times D_0=C_0\times D_1\times D_2=C\times D_2$ and 
\[d(C)\le d(C_0)+1\le|B_0|+1\le|B|\]
as desired.
\end{proof}

\begin{proof}[Proof of Theorem~A]
By \eqref{aut} it suffices to bound $|\Z(G)|$ in terms of $n$. Let $g_1,\ldots,g_m\in G$ be representatives for the cosets of $G/\Z(G)G'$. Let $U:=\langle g_1,\ldots,g_m\rangle G'$. Then 
\[d(U/G')\le m=|G:\Z(G)G'|\le|G:\Z(G)|\le n.\]
%Then 
%\[|U|=|U/G'||G'|\le\exp(U/G')^{d(U/G')}n^{2n^3}\le\exp(G)^mn^{2n^3}\]
%Let $\Phi(G/G')=N/G'$. Then $G/N$ is a direct product of elementary abelian groups. Hence, there exists a decomposition $G/N=\Z(G)N/N\times U/N$ such that 
%\[d(U/G')\le d(U/N)\le |U/N|=|G/\Z(G)N|\le|G/\Z(G)|\le n.\]
By \autoref{schur} and \eqref{dG},
\[|U|=|U/G'||G'|\le\exp(U/G')^{d(U/G')}n^{2n^3}\le\exp(G)^nn^{2n^3}.\]
Hence by \autoref{exp}, $|U|$ is bounded by a function on $n$.
By \autoref{lemabel2} we have $\Z(G)=C\times D$ such that $U\cap\Z(G)\le C$ and $d(C)\le|U\cap\Z(G)|\le|U|$. 
Now also $|C|$ is bounded and it remains to prove that $|D|$ can be bounded in terms of $n$.
Let $d=uc\in UC\cap D$ with $u\in U$ and $c\in C$. Then $u=dc^{-1}\in U\cap\Z(G)\le C$ and it follows that $d=dc^{-1}c\in D\cap C=1$. This shows
\[G=U\Z(G)=U(C\times D)=UC\times D.\]
Since every automorphism of $D$ extends to $G$, we may assume that $G=D$ is abelian. By \autoref{primes} we may assume that $G=G_p$ is a $p$-group, say 
\[G=\langle x_1\rangle\times\ldots\times\langle x_k\rangle\]
with $|\langle x_1\rangle|\ge\ldots\ge|\langle x_k\rangle|$. It is easily checked that the map
\begin{align*}
x_1\mapsto x_1x_l,&&x_i\mapsto x_i\qquad(2\le i\le k)
\end{align*}
defines an automorphism of $G$ whenever $2\le l\le k$. Hence, $k\le n$ and $|G|$ is bounded in terms of $n$ by \autoref{exp}.
%Nach Schur genügt es zu zeigen, dass $|G/G'|$ durch eine Funktion in $n$ beschränkt ist. Wegen $|G:\Z(G)G'|\le|G:\Z(G)|\le n$ genügt es $|\Z(G)G'/G'|$ abzuschätzen. 
%Nach \autoref{exp} reicht es $d(\Z(G)G'/G')$ durch eine Funktion in $n$ abzuschätzen. Sei $P\in\Syl_p(\Z(G)G'/G')$. Wegen $p\le n+1$ genügt es $d(P)$ abzuschätzen. Nach \autoref{lemabel2} ist $G/G'=C\times D$ mit $P\le C$ und $d(C)\le|P|$. 
\end{proof}

\section{The reverse bound}
As promised at the very beginning, we now give an optimal bound on $|\Aut(G)|$ in terms of $|G|$. Recall that a group $G$ is called \emph{boolean} if $\exp(G)\le 2$. In this case $G$ is abelian, since $gh=(gh)^{-1}=h^{-1}g^{-1}=hg$ for all $g,h\in G$. The following improves \eqref{easy}.

\begin{Prop}
For every finite group $G$ we have $d(G)\le\log_2|G|$ and
\[|\Aut(G)|\le\prod_{k=0}^{d(G)-1}\bigl(|G|-2^k\bigr)\]%\le\prod_{k=1}^{\lfloor\log_2|G|\rfloor}\bigl(|G|-2^{k-1}\bigr)\]
with equality if and only if $|G|$ is a prime or $G$ is boolean.
%Equality holds for $G=C_2\times\ldots\times C_2$.
\end{Prop}
\begin{proof}
If $G=1$, then $d(G)=0$ and equality holds by interpreting the empty product as $1$ (note that the trivial group is boolean). Now let $G\ne 1$ with a minimal generating set $g_1,\ldots,g_d\in G$ where $d=d(G)$. For $\alpha\in\Aut(G)$, also $\alpha(g_1),\ldots,\alpha(g_d)$ is a (minimal) generating set and $\alpha$ is uniquely determined by those images. Since $\alpha(g_1)\ne 1$, there are at most $|G|-1$ choices for $\alpha(g_1)$. Since $\alpha(g_2)\notin\langle\alpha(g_1)\rangle$, there are at most $|G\setminus\langle\alpha(g_1)\rangle|\le|G|-2$ possibilities for $\alpha(g_2)$ and so on. This proves $d(G)\le\log_2|G|$ and the inequality on $|\Aut(G)|$.

If equality holds, then for every $g\ne 1$ there exists an automorphism mapping $g_1$ to $g$. In particular, all non-trivial elements of $G$ have the same order, which necessarily must be a prime $p$ (if not, consider a power of $g$). If additionally $d=1$, then $|G|=|\langle g_1\rangle|=p$. On the other hand, if $d\ge 2$, then there are $|G|-2=|G\setminus\langle\alpha(g_1)\rangle|$ choices for $\alpha(g_2)$. Hence $p=|\langle\alpha(g_1)\rangle|=2$ and $G$ is boolean.

Conversely, every group of prime order $p$ has $p-1$ automorphisms by \eqref{aut}. Moreover, every boolean group $G$ is an $\FF_2$-vector space and $\Aut(G)\cong\GL(d,2)$ where $d=d(G)$. Counting matrices with linearly independent rows yields the well-known formula \[|\GL(d,2)|=(2^d-1)(2^d-2)\ldots(2^d-2^{d-1}).\]
Thus, we have shown equality.
\end{proof}

The proof above actually shows slightly more: If $|G|=p_1\ldots p_n$ with primes $p_1\le\ldots\le p_n$, then $d(G)\le n$ and
\[|\Aut(G)|\le\prod_{k=0}^{d(G)-1}\bigl(|G|-p_1\ldots p_k\bigr).\]
%more precisely that $d(G)$ is bounded by the number of prime divisors of $|G|$ counting multiplicities.

\section{Some related conjectures}

A complete classification of all finite groups with less than $48$ automorphisms was given by MacHale and Sheehy~\cite{MacHaleSheehy} (see also \cite{OEIS3}). They noticed that $\phi(|G|)\le|\Aut(G)|$ holds in these small cases where $\phi$ is Euler's totient function. In fact, this inequality was conjectured in general by Deaconescu~\cite{Deaconescu} who also conjectured that equality holds if and only if $G$ is cyclic (it is Problem 15.43 in the Kourovka Notebook~\cite{Kourovka}). If true, this would yield a bound on $|G|$ as well (e.g., $|G|\le|\Aut(G)|^{1+\epsilon}$ provided $|G|$ is large enough with respect to $\epsilon>0$). However, Bray and Wilson~\cite{BrayWilson,BrayWilson2} constructed solvable and nonsolvable counterexamples.

Similarly, the long-standing Problem~12.77 in \cite{Kourovka} proposed that $|G|$ divides $|\Aut(G)|$ for every nonabelian $p$-group $G$. This was disproved recently by González-Sánchez and Jaikin-Zapirain~\cite{SchenkmanCon} using pro-$p$ group techniques. In fact, $|\Aut(G)|/|G|$ can be arbitrarily small.

Yet another conjecture, this time from \cite{MacHaleSheehy}, reads $|G|\le|\End(G)|$ where $\End(G)$ is the set of endomorphisms of $G$. However, the triple cover $G=3.A_7$ of the alternating group of degree $7$ is a counterexample. Since $A_7$ is a simple group, $G$ has only three normal subgroups: $1$, $\Z(G)$ and $G$. Here, $\Z(G)$ cannot occur as a kernel of an endomorphism, because as a perfect group $G$ does not contain subgroups of index $3$. Hence, every nontrivial endomorphism is an automorphism. Moreover, it is known that $\Aut(G)$ acts faithfully on $G/\Z(G)\cong A_7$ (this holds for any quasisimple group). Since $\Aut(A_7)$ is isomorphic to the symmetric group $S_7$, we finally conclude that
\[|\End(G)|=1+|\Aut(G)|\le1+|\Aut(G/\Z(G))|=1+|S_7|=1+7!<\frac{3}{2}7!=|G|.\]

\section*{Acknowledgment}
The author is supported by the German Research Foundation (\mbox{SA 2864/1-2} and \mbox{SA 2864/3-1}).

\end{document}